\documentclass[11pt]{amsart}
\usepackage{graphicx, color}
\usepackage{amscd}
\usepackage{amsmath}
\usepackage{amsfonts}
\usepackage{amssymb}
\usepackage{mathrsfs}
\newtheorem{theorem}{Theorem}[section]
\newtheorem{lemma}[theorem]{Lemma}

\newtheorem{remark}[theorem]{Remark}

\newtheorem{step}{Step}
\allowdisplaybreaks
\numberwithin{equation}{section}

\numberwithin{equation}{section}
\newcommand{\RR}{\mathbb R}
\newcommand{\NN}{\mathbb N}

\renewcommand{\le}{\leqslant}
\renewcommand{\leq}{\leqslant}
\renewcommand{\ge}{\geqslant}
\renewcommand{\geq}{\geqslant}

\begin{document}

\title[Multiplicity results]{Multiplicity results for fractional Laplace problems with critical growth}

\thanks{The first author was supported by {\it Coordena\c c\~ao de Aperfei\c conamento de pessoal de n\'ivel superior} (CAPES) through the fellowship 33003017003P5--PNPD20131750--UNICAMP/MATEM\'ATICA. The second and the third author were supported by the INdAM-GNAMPA Project 2016 {\it Problemi variazionali su variet\`a riemanniane e gruppi di Carnot}, by the DiSBeF Research Project 2015 {\it Fenomeni non-locali: modelli e applicazioni} and by the DiSPeA Research Project 2016 {\it Implementazione e testing di modelli di fonti energetiche ambientali per reti di sensori senza fili autoalimentate}. The third author was supported by the ERC grant $\epsilon$ ({\it Elliptic Pde's and Symmetry of Interfaces and Layers for Odd Nonlinearities}).}

\author[A. Fiscella]{Alessio Fiscella}
\address{Departamento de Matem\'atica, Universidade Estadual de Campinas, IMECC,
Rua S\'ergio Buarque de Holanda 651, SP CEP 13083--859 Campinas, BRAZIL}
\email{\tt fiscella@ime.unicamp.br}

\author[G. Molica Bisci]{Giovanni Molica Bisci}
\address{Dipartimento PAU,
          Universit\`a `Mediterranea' di Reggio Calabria,
          Via Melissari 24, 89124 Reggio Calabria, Italy}
\email{\tt gmolica@unirc.it}

\author[R. Servadei]{Raffaella Servadei}
\address{Dipartimento di Scienze Pure e Applicate (DiSPeA), Universita degli Studi di Urbino
`Carlo Bo', Piazza della Repubblica 13, 61029 Urbino (Pesaro e Urbino), Italy}
\email{\tt raffaella.servadei@uniurb.it}



\begin{abstract}
This paper deals with multiplicity and bifurcation results for nonlinear problems driven by the fractional Laplace operator $(-\Delta)^s$ and involving a critical Sobolev term. In particular, we consider
$$\left\{
\begin{array}{ll}
(-\Delta)^su=\gamma\left|u\right|^{2^*-2}u+f(x,u) &  \mbox{in } \Omega\\
u=0 & \mbox{in } \RR^n\setminus \Omega,
\end{array}\right.$$
where $\Omega\subset\mathbb R^n$ is an open bounded set with continuous boundary, $n>2s$ with $s\in(0,1)$, $\gamma$ is a positive real parameter, $2^*=2n/(n-2s)$ is the fractional critical Sobolev exponent and $f$ is a Carath\'{e}odory function satisfying different subcritical conditions.

\end{abstract}

\maketitle

\section{Introduction}
Recently, the interest towards nonlocal fractional Laplacian equations involving a critical term has grown more and more.
Concerning the existence result for this kind of problems, a positive answer has been given in the recent papers \cite{colorado2, fv, ms, sY, servadeivaldinociBN, servadeivaldinociBNLOW}: in all these works well known existence results for classical Laplace operators were extended to the nonlocal fractional setting. A natural question is to ask when it is possible to get more than a non--trivial solution, giving a multiplicity result. In literature few attempts have been made to answer this question. In particular we refer to very recent papers \cite{fms, pereira} which give a bifurcation result.

Motivated by the above papers, here we deal with the following problem
\begin{equation}\label{P}
\left\{\begin{array}{ll}
(-\Delta)^su=\gamma\left|u\right|^{2^*-2}u+f(x,u) & \mbox{in } \Omega\\
u=0 & \mbox{in } \mathbb{R}^{n}\setminus\Omega,
\end{array}
\right.
\end{equation}
where $s\in(0,1)$ is fixed, $n>2s$, $\Omega\subset\mathbb{R}^{n}$ is an open and bounded set with continuous boundary, $2^*=2n/(n-2s)$ and $(-\Delta)^s$ is the fractional Laplace operator, that may be defined (up to a normalizing constant) by the Riesz potential as follows
\begin{equation}\label{delta}
\mathcal (-\Delta)^su(x)=\int_{\mathbb{R}^{n}}\frac{2u(x)-u(x+y)-u(x-y)}{\left|y\right|^{n+2s}}dy,\quad x\in\mathbb R^n\,,
\end{equation}
as defined in \cite{valpal} (see this paper and the references therein for further details on fractional Laplacian).

Concerning the nonlinearity in \eqref{P}, in the present work we assume that $f:\Omega\times\mathbb R\rightarrow\mathbb R$ is a {\em Carath\'{e}odory function} satisfying the following condition
\begin{equation}\label{f1}
\sup\Big\{\left|f(x,t)\right|:\;\;x\in\Omega,\;\;\left|t\right|\leq M\Big\}<+\infty \textit{ for any }M>0.
\end{equation}

The main aim of the present paper is to establish bifurcation results for $\eqref{P}$. For this, we need that $f(x,t)$ is odd in $t$, i.e.
\begin{equation}\label{odd}
f(x,t)=-f(x,-t)\quad \mbox{for any}\,\,\, t\in \RR \mbox{ and a.e.}\,\,\,x\in \Omega\,,
\end{equation}
in order to apply the symmetric version of the Mountain Pass Theorem due to Ambrosetti and Rabinowitz (see \cite{ar}). However, with respect to the classical case presented in \cite{ar}, we use a weaker condition than the usual one of Ambrosetti-Rabinowitz, in order to overcome the lack of compactness at critical level $L^{2^*}(\Omega)$. Thus, we assume that $f$ and its primitive $F$, defined as
\begin{equation}\label{F}
F(x,t)=\int^t_0f(x,\tau)d\tau\,,
\end{equation}
satisfy
\begin{equation}\label{f2}
\lim_{\left|t\right|\rightarrow +\infty}\frac{f(x,t)}{\left|t\right|^{2^*-1}}=0 \textit{ uniformly a.e. in }\Omega;
\end{equation}
\begin{equation}\label{f3}
\begin{gathered}\vspace{-.2cm}
\textit{there exist } \sigma\in[0,2)\textit{ and }a_1, a_2>0\textit{ such that}\\
\frac{1}{2}\,f(x,t)t-F(x,t)\geq-a_1-a_2\left|t\right|^\sigma \textit{ for any }t\in\mathbb{R} \textit{ and a.e. } x\in\Omega;\end{gathered}
\end{equation}
\begin{equation}\label{f4}
\begin{gathered}\vspace{-.2cm}
\textit{there exist } \theta\in(2,2^*)\textit{ and }b_1, b_2>0\textit{ such that}\\
F(x,t)\leq b_1\left|t\right|^\theta+b_2 \textit{ for any }t\in\mathbb{R}\textit{ and a.e. }x\in\Omega;
\end{gathered}
\end{equation}
\begin{equation}\label{f5}
\begin{gathered}\vspace{-.2cm}
\textit{there exist } c_1>0, h_1\in L^1(\Omega)\textit{ and }\Omega_0\subset\Omega\textit{ with }\left|\Omega_0\right|>0\textit{ such that}\\
F(x,t)\geq -h_1(x)\left|t\right|^2-c_1 \textit{ for any }t\in\mathbb{R}\textit{ and a.e. }x\in\Omega\textit{ and}\\
\liminf_{\left|t\right|\rightarrow +\infty}\frac{F(x,t)}{\left|t\right|^2}=+\infty \textit{ uniformly a.e. in }\Omega_0.
\end{gathered}
\end{equation}
We are now ready to state our first result.

\begin{theorem}\label{a}
Let $s\in (0,1)$, $n>2s$, $\Omega$ be an open bounded subset of $\RR^n$ with continuous boundary, and let $f$ be a function satisfying assumptions \eqref{f1}, \eqref{odd}, \eqref{f2}, \eqref{f3}--\eqref{f5}.

Then, for any $k\in\mathbb{N}$ there exists $\gamma_k\in(0,+\infty]$ such that \eqref{P} admits at least $k$ pairs of non--trivial solutions for any $\gamma\in(0,\gamma_k)$.
\end{theorem}

In the next result we establish a multiplicity result of solutions for \eqref{P} without assuming that primitive $F$ still satisfies a general subcritical growth like in \eqref{f4}. However, we need a stronger condition than \eqref{f5}. That is, given $j, k\in\mathbb{N}$ with $j\leq k$, we consider these different versions of \eqref{f4} and \eqref{f5}
\begin{equation}\label{ft5}
\begin{gathered}\vspace{-.2cm}
\textit{there exists a measurable function }a:\Omega\rightarrow\mathbb R\textit{ such that}\\
\limsup_{t\rightarrow0}2\frac{F(x,t)}{\left|t\right|^2}=a(x)\textit{ uniformly a.e. in }\Omega,\\
a(x)\leq\lambda_j\textit{ a.e. in }\Omega\textit{ and }a(x)<\lambda_j\textit{ on a set of positive measure contained in } \Omega;
\end{gathered}
\end{equation}
\begin{equation}\label{ft4}
\begin{gathered}\vspace{-.2cm}
\textit{there exists } B>0\textit{ such that}\\
F(x,t)\geq \lambda_k\frac{\left|t\right|^2}{2}-B \textit{ for any }t\in\mathbb{R}\textit{ and a.e. }x\in\Omega,\end{gathered}
\end{equation}
where $\lambda_j\leq\lambda_k$ are eigenvalues of $(-\Delta)^s$, as recalled in Section \ref{sec main}.

With the above conditions we still can apply the Mountain Pass Theorem given in \cite{ar}, getting the following result:
\begin{theorem}\label{b}
Let $s\in (0,1)$, $n>2s$, $\Omega$ be an open bounded subset of $\RR^n$ with continuous boundary. Let $j$, $k\in\mathbb{N}$, with $j\leq k$, and let $f$ be a function satisfying assumptions \eqref{f1}, \eqref{odd}, \eqref{f2}, \eqref{f3}, \eqref{ft5} and \eqref{ft4}.

Then, there exists $\gamma_{k,j}\in(0,+\infty]$ such that \eqref{P} admits at least $k-j+1$ pairs of non--trivial solutions for any $\gamma\in(0,\gamma_{k,j})$.
\end{theorem}

A natural question is to investigate what happens when $f$ has not any symmetry. In this case it is still possible to get a multiplicity result, by studying two truncated problems related to \eqref{P}. These auxiliary problems are still variational and by using the Mountain Pass Theorem we get at least two solutions of different sign for them, as stated in the following result:
\begin{theorem}\label{c}
Let $s\in (0,1)$, $n>2s$, $\Omega$ be an open bounded subset of $\RR^n$ with continuous boundary. Let $f$ satisfy $f(x,0)=0$, \eqref{f1}, \eqref{f2}, \eqref{f3}, \eqref{ft5} and \eqref{ft4} with $j=k=1$.

Then, there exists $\gamma_1>0$ such that \eqref{P} admits a non--trivial non--negative and a non--trivial non--positive solution for any $\gamma\in(0,\gamma_1)$.
\end{theorem}

The main tools used in order to prove Theorem~\ref{a}--Theorem~\ref{c} are variational and topological methods and a suitable decomposition of the functional space $X_0^s(\Omega)$ where we look for solutions of problem~\eqref{P}, through the eigenvalues of the fractional Laplace operator.

An interesting open problem is to prove the main results of the present paper in a more general framework, like the one given in the following problem:
\begin{equation}\label{Pgen}
\left\{\begin{array}{ll}
-\mathcal L^p_K u=\gamma\left|u\right|^{p^*-2}u+f(x,u) & \mbox{in } \Omega\\
u=0 & \mbox{in } \mathbb{R}^{n}\setminus\Omega\,,
\end{array}
\right.
\end{equation}
where $\Omega\subset\mathbb R^n$ is an open and bounded set with continuous boundary, $n>ps\geq 2s$, $p^*=pn/(n-ps)$ and $\mathcal L^p_K$ is a nonlocal operator defined as follows:
$$\mathcal L^p_K u(x)=2\lim_{\varepsilon\searrow0}\int_{\mathbb{R}^{n}\setminus B_{\varepsilon}(x)}|u(x)-u(y)|^{p-2}(u(x)-u(y))K(x-y)dy,\quad x\in\mathbb R^n\,.$$
Here, the kernel $K:\mathbb{R}^{n}\setminus \left\{0\right\}\rightarrow(0,+\infty)$ is a
{\em measurable function} for which
\begin{equation}\label{K1}
mK\in L^{1}(\mathbb{R}^{n}),\textit{ with }\; m(x)=\min\left\{\left|x\right|^{p},1\right\};
\end{equation}
\begin{equation}\label{K2}
\textit{there exists}\;\; \theta >0\;\textit{such that }\;K(x)\geq\theta\left|x\right|^{-(n+ps)}\;\textit{for any } x\in\mathbb{R}^{n}\setminus\{0\},
\end{equation}
hold true. A model for $\mathcal L^p_K$ is given by the fractional $p$-Laplacian $(-\Delta)^s_p$ which (up to normalization factors) may be defined for any $x\in\mathbb{R}^{n}$ as
$$(-\Delta)^s_p u(x)=2\lim_{\varepsilon\searrow0}\int_{\RR^n\setminus B_{\varepsilon}(x)}\frac{\left|u(x)-u(y)\right|^{p-2}(u(x)-u(y))}{|x-y|^{n+sp}}\,dxdy\,.$$

For problem~\eqref{Pgen} the appropriate functional space where finding solution is $X^{s,\,p}_0(\Omega)$, defined as
$$X^{s,\,p}_0(\Omega)=\{g\in X^{s,\,p}(\Omega) : g=0 \mbox{ a.e. in }
\mathbb R^n\setminus\Omega\}.$$
Here $X^{s,\,p}(\Omega)$ denotes the linear space of Lebesgue
measurable functions $u:\RR^n\rightarrow\RR$ whose restrictions to $\Omega$ belong to $L^p(\Omega)$ and such that
	\[\mbox{the map }
(x,y)\mapsto (u(x)-u(y))^p K(x-y) \mbox{ is in } L^1\big(Q,dxdy\big),
\]
where $Q=\mathbb R^n \times \mathbb R^n\setminus\left((\mathbb R^n\setminus\Omega)\times(\mathbb R^n\setminus\Omega)\right)$. It is immediate to see that $X^{s,\,p}_0(\Omega)$ is a Banach space endowed with the following norm
\begin{equation}\label{norma2}
\|u\|_{s,p}=
\Big(\iint_{\mathbb R^n \times \mathbb R^n}
|u(x)-u(y)|^pK(x-y)\,dx\,dy\Big)^{1/p}\,.
\end{equation}
When $p=2$ and $K(x)=\left|x\right|^{-(n+2s)}$ the space $X^{s,\,p}_0(\Omega)$ coincides with $X_0^s(\Omega)$ defined in \eqref{spazio} (see \cite[Lemma~5]{svmountain}). In such a case the statements of Theorem~\ref{a}--Theorem~\ref{c} are still valid and their proofs can be performed exactly with the same arguments considered in the model case of the fractional Laplace operator $(-\Delta)^s$.

In order to treat problem~\eqref{Pgen} when $p\not =2$ we have to adapt in a suitable way the arguments used for studying \eqref{P}. Indeed, in this case the main difficulty is related to the fact that we have to understand how to decompose the space $X^{s,\,p}_0(\Omega)$. Indeed, when $p\not=2$ the full spectrum of $(-\Delta)^s_p$ and of $\mathcal L^p_K$ is still almost unknown, even if some important properties of the first eigenvalue and of the higher order (variational) eigenvalues have been established in \cite{fp, ll}.
We would recall that in \cite{bartolomolica} the authors proposed a definition of quasi--eigenvalues for $(-\Delta)^s_p$ and using them considered a suitable decomposition of $X^{s,\,p}_0(\Omega)$ which turns out to be the known one for $p=2$.

The paper is organized as follows.
In Section~\ref{sec variational} we introduce the variational
formulation of the problem under consideration. Section~\ref{sec palais} is devoted to the proof of the compactness property for problem~\eqref{P}.
In Section~\ref{sec main} we conclude the proofs of Theorem~\ref{a}--Theorem~\ref{c}.

\section{Variational setting}\label{sec variational}
Problem~\eqref{P} has a variational structure and the natural space where finding solutions is the homogeneous fractional Sobolev space $H^s_0(\Omega)$. In order to study \eqref{P} it is important to encode the ``boundary condition'' $u=0$ in
$\mathbb{R}^n\setminus\Omega$ (which is different from the classical case of the Laplacian, where it is required $u=0$ on $\partial \Omega$) in the weak formulation,
by considering also that the interaction between $\Omega$ and its complementary in $\mathbb{R}^n$ gives a positive contribution
in the so-called {\em Gagliardo norm} given as
\begin{equation}\label{norma}
\left\|u\right\|_{H^s(\mathbb R^n)}=\left\|u\right\|_{L^2(\mathbb R^n)}+\Big(\iint_{\mathbb R^n \times \mathbb R^n} \frac{|u(x)-u(y)|^2}{\left|x-y\right|^{n+2s}}\,dx\,dy\Big)^{1/2}.
\end{equation}

The functional space that takes into account this boundary condition will be denoted by $X_0^s(\Omega)$ and it is defined as
\begin{equation}\label{spazio}
X_0^s(\Omega)=\big\{u\in H^s(\mathbb R^n):\,u=0\mbox{ a.e. in } \mathbb R^n\setminus \Omega\big\}.
\end{equation}
We refer to \cite{svmountain, servadeivaldinociBN} for a general definition of $X_0^s(\Omega)$ and its properties.
We also would like to point out that, when $\partial\Omega$ is continuous, by \cite[Theorem~6]{fsv} the space $X_0^s(\Omega)$ can be seen as the closure of $C^\infty_0(\Omega)$ with respect to the norm \eqref{norma}. This last point will play a crucial role in the proof of the compactness condition for the energy functional related to \eqref{P}.

In $X_0^s(\Omega)$ we can consider the following norm
\begin{equation}\label{normax0}
\left\|u\right\|_{X_0^s(\Omega)}=\Big(\iint_{\mathbb R^n \times \mathbb R^n} \frac{|u(x)-u(y)|^2}{\left|x-y\right|^{n+2s}}\,dx\,dy\Big)^{1/2},
\end{equation}
which is equivalent to the usual one defined in \eqref{norma} (see \cite[Lemma~6]{svmountain}).
We also recall that $(X_0^s(\Omega),\left\|\cdot\right\|_{X_0^s(\Omega)})$ is a Hilbert space, with the scalar product defined as
\begin{equation}\label{prodottoscalare}
\left\langle u,v\right\rangle_{X_0^s(\Omega)}=\iint_{\mathbb R^n \times \mathbb R^n} \frac{(u(x)-u(y))(v(x)-v(y))}{\left|x-y\right|^{n+2s}}\,dx\,dy.
\end{equation}

{\it From now on, in order to simplify the notation, we will denote $\|\cdot\|_{X_0^s(\Omega)}$ and $\left\langle \cdot,\cdot\right\rangle_{X_0^s(\Omega)}$ by $\|\cdot\|$ and $\left\langle \cdot,\cdot\right\rangle$ respectively, and $\|\cdot\|_{L^q(\Omega)}$ by $\|\cdot\|_q$ for any $q\in[1,+\infty]$.}

A function $u\in X_0^s(\Omega)$ is said to be a ({\em weak}) {\em solution of problem}~\eqref{P} if $u$ satisfies the following weak formulation
\begin{equation}\label{wf}
\left\langle u,\varphi\right\rangle
=\displaystyle\gamma\int_{\Omega} \left|u(x)\right|^{2^*-2}u(x)\varphi(x)dx+\int_{\Omega} f(x,u(x))\varphi(x) dx,
\end{equation}
for any $\varphi \in X_0^s(\Omega)$.
We observe that \eqref{wf} represents the Euler--Lagrange equation of the functional~$\mathcal J_\gamma:X_0^s(\Omega)\to \RR$ defined as
\begin{equation}\label{Jgamma}
\mathcal J_\gamma(u)=\frac 1 2 \left\|u\right\|^2-\frac{\gamma}{2^*}\left\|u\right\|^{2^*}_{2^*}
-\int_\Omega F(x,u(x))\,dx\,,
\end{equation}
where $F$ is as in \eqref{F}. It is easily seen that $\mathcal J_\gamma$ is well defined thanks to \eqref{f1}--\eqref{f2} and \cite[Lemma~6]{svmountain}.
Moreover, $\mathcal J_\gamma\in C^1(X_0^s(\Omega))$, thus critical points of $\mathcal J_\gamma$ are solutions to
problem~\eqref{wf}, that is weak solutions for \eqref{P}.

The proofs of Theorem~\ref{a} and Theorem~\ref{b} are mainly based on variational and topological methods.
Precisely, here we will perform the following version of the symmetric Mountain Pass Theorem (see \cite{ar, bartolo, silva}).

\begin{theorem}[Abstract critical point theorem]\label{abs}
Let $E=V\oplus X$, where $E$ is a real Banach space and $V$ is finite dimensional. Suppose that $\mathcal I\in C^1(E,\mathbb{R})$ is a functional satisfying the following conditions:
\begin{itemize}
\item [$(I_1)$] $\mathcal I(u)=\mathcal I(-u)$ and $\mathcal I(0)=0$;
\item [$(I_2)$] there exists a constant $\rho>0$ such that $\mathcal I|_{\partial B_\rho\cap X}\geq 0$;
\item [$(I_3)$] there exists a subspace $W\subset E$ with $\mbox{dim } V<\mbox{dim } W<+\infty$ and there is $M>0$ such that ${\displaystyle \max_{u\in W}\mathcal I(u)<M}$;
\item [$(I_4)$] considering $M>0$ from $(I_3)$, $\mathcal I(u)$ satisfies $(PS)_c$ condition for $0\leq c\leq M$.
\end{itemize}

Then, there exist at least $\mbox{dim }W-\mbox{codim }V$ pairs of non--trivial critical points of $\mathcal I$.
\end{theorem}

In order to prove our main results, the idea consists in applying Theorem~\ref{abs} to the functional $\mathcal J_\gamma$. At this purpose
note that when $f$ is odd in $t$, $\mathcal J_\gamma$ is even and also $\mathcal J_\gamma(0)=0$.
{\it Thus, condition~$(I_1)$ of Theorem~\ref{abs} is always verified by $\mathcal J_\gamma$ and we will not recall it in the sequel.}

For the proof of Theorem~\ref{c} we will use the following version of the Mountain Pass Theorem (see \cite{silva}):
\begin{theorem}\label{abs2}
Let $E$ be a real Banach space. Suppose that $\mathcal I\in C^1(E,\mathbb{R})$ is a functional satisfying the following conditions:
\begin{itemize}
\item [$(I_1)$] $\mathcal I(0)=0$;
\item [$(I_2)$] there exists a constant $\rho>0$ such that $\mathcal I|_{\partial B_\rho}\geq 0$;
\item [$\widehat{(I_3)}$] there exist $v_1\in\partial B_1$ and $M>0$ such that ${\displaystyle \sup_{t\geq0}\mathcal I(tv_1)\leq M}$;
\item [$(I_4)$] considering $M>0$ from $(I_3)$, $\mathcal I(u)$ satisfies $(PS)_c$ condition for $0\leq c\leq M$.
\end{itemize}

Then, $\mathcal I$ possesses a non--trivial critical point.
\end{theorem}

\section{The Palais--Smale condition}\label{sec palais}
In this section we verify that the functional $\mathcal J_\gamma$ satisfies the $(PS)_c$ condition under a suitable level. For this, we use some preliminary estimates concerning the nonlinearity $f$ and its primitive $F$. By \eqref{f1} and \eqref{f2} for any $\varepsilon>0$ there exists a constant $C_\varepsilon>0$ such that
\begin{equation}\label{sottocritico}
\left|f(x,t)t\right|\leq C_\varepsilon+\varepsilon\left|t\right|^{2^*}\quad\mbox{for any }t\in\mathbb{R}\mbox{ and a.e. }x\in\Omega
\end{equation}
and
\begin{equation}\label{sottocritico2}
\left|F(x,t)\right|\leq C_\varepsilon+\frac{\varepsilon}{2^*}\left|t\right|^{2^*}\quad\mbox{for any }t\in\mathbb{R}\mbox{ and a.e. }x\in\Omega.
\end{equation}

We recall that $\left\{u_{j}\right\}_{j\in\mathbb{N}}\subset X_0^s(\Omega)$ is a {\it Palais--Smale sequence for $\mathcal J_\gamma$ at level $c\in\mathbb{R}$} (in short $(PS)_c$ sequence) if
\begin{equation}\label{ps1}
\mathcal J_\gamma(u_j)\rightarrow c\quad\mbox{and}\quad\mathcal J'_\gamma(u_j)\rightarrow0\quad\mbox{as}\;j\rightarrow +\infty.
\end{equation}
We say that  $\mathcal J_\gamma$ {\it satisfies the Palais--Smale condition at level} $c$ if any Palais--Smale sequence
$\left\{u_{j}\right\}_{j\in\mathbb{N}}$ at level $c$ admits a convergent subsequence in $X_0^s(\Omega)$.

As usual, we first prove the boundedness of the $(PS)_c$ sequence.
\begin{lemma}\label{bound}
Let $f$ satisfy \eqref{f1}, \eqref{f2} and \eqref{f3}. For any $\gamma>0$, let $c>0$ and $\left\{u_{j}\right\}_{j\in\mathbb{N}}$ be a $(PS)_c$ sequence for $\mathcal J_\gamma$.

Then, $\left\{u_{j}\right\}_{j\in\mathbb{N}}$ is bounded in $X_0^s(\Omega)$.
\end{lemma}
\begin{proof}
Fix $\gamma>0$. By \eqref{ps1} there exists $C>0$ such that
\begin{equation}\label{3.10}
\left|\mathcal J_\gamma(u_j)\right|\leq C\quad\mbox{and}\quad\left| \mathcal J'_\gamma(u_j)\left(\frac{u_j}{\left\|u_j\right\|}\right)\right|\leq C
\quad\mbox{for any }j\in\mathbb{N}.
\end{equation}
Moreover, by \eqref{f3} and H\"{o}lder inequality we have
\begin{equation}\label{3.11}
\mathcal J_\gamma(u_j)-\frac{1}{2}\mathcal J'_\gamma(u_j)(u_j)\ge\frac{s\gamma}{n}\left\|u_j\right\|^{2^*}_{2^*}-a_1\left|\Omega\right|-a_2\left|\Omega\right|^{\frac{2^*-\sigma}{2^*}}\left\|u_j\right\|^{\sigma}_{2^*}.
\end{equation}

From Young's inequality with exponents $p=2^*/\sigma$ and $q=2^*/(2^*-\sigma)$ we also get
	\[\left\|u_j\right\|^{\sigma}_{2^*}\leq\delta\left\|u_j\right\|^{2^*}_{2^*}+C_\delta,
\]
for suitable $\delta$, $C_\delta>0$. The last inequality combined with \eqref{3.10} and \eqref{3.11} says that
\begin{equation}\label{3.12}
\left\|u_j\right\|^{2^*}_{2^*}\leq C'\Big(\left\|u_j\right\|+1\Big),
\end{equation}
for another positive constant $C'$.

Now, by \eqref{sottocritico2}, \eqref{3.10} and \eqref{3.12} we obtain
	\[C\geq\mathcal J_\gamma(u_j)\geq\frac{1}{2}\left\|u_j\right\|^2-\left(\frac{C'\gamma}{2^*}-\frac{C'\varepsilon}{2^*}\right)\left(1+\left\|u_j\right\|\right)-C_\varepsilon\left|\Omega\right|,
\]
which gives the boundedness of $\left\{u_{j}\right\}_{j\in\mathbb{N}}$ in $X_0^s(\Omega)$.
\end{proof}

Now, we can prove the relatively compactness of a $(PS)_c$ sequence under a suitable level. Here, we must pay attention to the lack of compactness at level $L^{2^*}(\Omega)$.
\begin{lemma}\label{palais}
Let $f$ satisfy \eqref{f1}, \eqref{f2} and \eqref{f3}.

Then, for any $M>0$ there exists $\gamma^*>0$ such that $\mathcal J_\gamma$ satisfies the $(PS)_c$ condition for any $c\leq M$, provided $0<\gamma<\gamma^*$.
\end{lemma}
\begin{proof}
Fix $M>0$. We set
\begin{equation}\label{gamma}
\gamma^*=\min\left\{S(n,s), \left[\left(S(n,s)\right)^{\frac{n}{2s}}\left(\frac{s}{n(M+A)}\right)^{\frac{2^*}{2^*-\sigma}}\right]^{\frac{1}{n/2s-2^*/(2^*-\sigma)}}\right\}
\end{equation}
with
\begin{equation}\label{A}
A=a_1\left|\Omega\right|+a_2\left|\Omega\right|^{\frac{2^*-\sigma}{2^*}}\,,
\end{equation}
where $a_1$, $a_2$, $\sigma$ are the constants given in \eqref{f3}, while $S(n,s)$ is the best constant of the fractional Sobolev embedding (see \cite[Lemma~6]{svmountain}) defined as
\begin{equation}\label{Sns}
S(n,s)=\inf_{v\in H^s(\mathbb R^n)\setminus\left\{0\right\}}\frac{\displaystyle{\iint_{\mathbb R^n \times \mathbb R^n}\frac{\left|v(x)-v(y)\right|^2}{\left|x-y\right|^{n+2s}}dx dy}}{\displaystyle{\left(\int_{\mathbb R^n}\left|v(x)\right|^{2^*}\right)^{2/2^*}}}>0.
\end{equation}

Given $\gamma<\gamma^*$ and $c<M$, let us consider a $(PS)_c$ sequence $\left\{u_{j}\right\}_{j\in\mathbb{N}}$ for $\mathcal J_\gamma$.
Since by Lemma~\ref{bound} we have that $\left\{u_{j}\right\}_{j\in\mathbb{N}}$ is bounded in $X_0^s(\Omega)$, by applying also \cite[Lemma~8]{sv} and \cite[Theorem~IV.9]{brezis}, there exists $u\in X_0^s(\Omega)$ such that, up to a subsequence,
\begin{equation}\label{ujconvX0}
u_j \rightharpoonup u \quad \mbox{weakly in}\,\, X_0^s(\Omega),
\end{equation}
\begin{equation}\label{ujconvLq}
u_j \to u \quad \mbox{in}\,\, L^q(\Omega),
\end{equation}
with $q\in[1,2^*)$ and
\begin{equation}\label{ujconvae}
u_j \to u \quad \mbox{a.e in}\,\, \Omega,
\end{equation}
as $j\to +\infty$.

Now, we claim that
\begin{equation}\label{claim}
\left\|u_j\right\|^2\to\left\|u\right\|^2\quad\mbox{as }j\to +\infty,
\end{equation}
which easily implies that $u_j\to u$ in $X_0^s(\Omega)$ as $j\to +\infty$, thanks to \eqref{ujconvX0}.

First of all, from Phrokorov's Theorem we deduce the existence of two positive measures $\mu$ and $\nu$ on $\mathbb{R}^n$ such that
\begin{equation}\label{convergenza misure}
\left|(-\Delta)^{s/2} u_j(x)\right|^2 dx\stackrel{*}{\rightharpoonup}\mu\quad\mbox{and}\quad\left|u_j(x)\right|^{2^*}dx\rightharpoonup\nu\quad\mbox{in }\mathcal M(\mathbb R^n)
\end{equation}
as $j\to +\infty$.
By \cite[Theorem~6]{fsv}, thanks to our assumptions on $\partial\Omega$, it is easy to see that $X_0^s(\Omega)$ can also be defined as the closure of $C^\infty_0(\Omega)$ with respect to the norm \eqref{norma}. Hence, $X_0^s(\Omega)$ is consistent with the functional space introduced in \cite{palatuccipisante}.
Thus, by \cite[Theorem~2]{palatuccipisante} we obtain an at most countable set of distinct points $\left\{x_i\right\}_{i\in J}$,
non--negative numbers $\left\{\nu_i\right\}_{i\in J}$, $\left\{\mu_i\right\}_{i\in J}$ and a positive measure $\widetilde{\mu}$,
with $Supp\; \widetilde{\mu}\subset\overline{\Omega}$, such that
\begin{equation}\label{nu}
\nu=\left|u(x)\right|^{2^*}dx+\sum_{i\in J} \nu_i\delta_{x_i}, \quad \mu=\left|(-\Delta)^{s/2} u(x)\right|^2 dx+\widetilde{\mu}+\sum_{i\in J}\mu_i\delta_{x_i},
\end{equation}
and
\begin{equation}\label{mu}
\nu_i\leq S(n,s)^{-\frac{2^*}{2}}\mu^{\frac{2^*}{2}}_{i}
\end{equation}
for any $i\in J$, where $S(n,s)$ is the constant given in \eqref{Sns}.
Now, in order to prove \eqref{claim} we proceed by steps.

\begin{step}\label{step1}
Fix $i_0\in J$. Then, either $\nu_{i_0}=0$ or
\begin{equation}\label{step}
\nu_{i_0}\geq\displaystyle\left[\frac{S(n,s)}{\gamma}\right]^{n/2s}.
\end{equation}
\end{step}
\begin{proof}
Let $\psi\in C^{\infty}_{0}(\mathbb{R}^n,[0,1])$ be such that $\psi\equiv 1$ in $B(0,1)$ and $\psi\equiv0$ in $\mathbb{R}^n\setminus B(0,2)$. For any $\delta\in(0,1)$ we set $$\psi_{\delta,i_0}(x)=\psi\Big((x-x_{i_0})/\delta\Big)\,.$$
Clearly the sequence $\left\{\psi_{\delta,i_0} u_j\right\}_{j\in\mathbb{N}}$ is bounded in $X_0^s(\Omega)$ by Lemma~\ref{bound}, and so by \eqref{ps1} it follows that
$$\mathcal J'_\gamma(u_j)(\psi_{\delta,i_0} u_j)\to 0$$ as $j\to +\infty$. In other words
\begin{equation}\label{that is}
\begin{alignedat}2
o(1)+&\iint_{\mathbb R^n \times \mathbb R^n}\frac{\big(u_j(x)-u_j(y)\big)\big(\psi_{\delta,i_0}(x)u_j(x)-\psi_{\delta,i_0}(y)u_j(y)\big)}
{\left|x-y\right|^{n+2s}}\,dxdy\\
&=\gamma\int_\Omega \left|u_j(x)\right|^{2^*}\psi_{\delta,i_0}(x)dx+\int_\Omega f(x, u_j(x))\psi_{\delta,i_0}(x)u_j(x)dx,
\end{alignedat}
\end{equation}
as $j\to +\infty$.

By \cite[Proposition~3.6]{valpal} and taking into account the definition of $(-\Delta)^s$ given in \eqref{delta}, we know that for any $v\in X_0^s(\Omega)$
$$\iint_{\mathbb R^n \times \mathbb R^n}\frac{|v(x)-v(y)|^2}{\left|x-y\right|^{n+2s}}\, dxdy=\int_{\mathbb{R}^n}\left|(-\Delta)^{s/2}v(x)\right|^2dx\,.$$
By taking derivative of the above equality, for any $v, w\in X_0^s(\Omega)$ we obtain
\begin{equation}\label{deriv}
\iint_{\mathbb R^n \times \mathbb R^n} \frac{(v(x)-v(y))(w(x)-w(y))}{\left|x-y\right|^{n+2s}}dxdy=\int_{\mathbb{R}^n}(-\Delta)^{s/2}v(x)(-\Delta)^{s/2}w(x)dx.
\end{equation}
Furthermore, for any $v, w\in X_0^s(\Omega)$ we have
\begin{equation}\label{prod}
(-\Delta)^{s/2}(v w)=v(-\Delta)^{s/2}w+w(-\Delta)^{s/2}v-2I_{s/2}(v,w),
\end{equation}
where the last term is defined, in the principal value sense, as follows
	\[I_{s/2}(v,w)(x)=P.V.\int_{\mathbb{R}^n}\frac{(v(x)-v(y))(w(x)-w(y))}{\left|x-y\right|^{n+s}}\,dy
\]
for any $x\in\mathbb R^n$.

Thus, by \eqref{deriv} and \eqref{prod} the integral in the left--hand side of \eqref{that is} becomes
\begin{equation}\label{that2}
\begin{aligned}
&\iint_{\mathbb R^n \times \mathbb R^n}\frac{\big(u_j(x)-u_j(y)\big)\big(\psi_{\delta,i_0}(x)u_j(x)-\psi_{\delta,i_0}(y)u_j(y)\big)}{\left|x-y\right|^{n+2s}}\,dxdy\\
& \quad = \int_{\mathbb{R}^n}(-\Delta)^{s/2}u_j(x)(-\Delta)^{s/2}(\psi_{\delta,i_0}u_j)(x)dx\\
& \quad =\int_{\mathbb{R}^n}u_j(x)(-\Delta)^{s/2}u_j(x)(-\Delta)^{s/2}\psi_{\delta,i_0}(x)dx\\
&\quad\quad+\int_{\mathbb{R}^n}\left|(-\Delta)^{s/2}u_j(x)\right|^2\psi_{\delta,i_0}(x)dx\\
&\quad\quad-2\int_{\mathbb{R}^n}(-\Delta)^{s/2}u_j(x)\int_{\mathbb{R}^n}\frac{(u_j(x)-u_j(y))(\psi_{\delta,i_0}(x)-\psi_{\delta,i_0}(y))}{\left|x-y\right|^{n+s}}dxdy.
\end{aligned}
\end{equation}
By \cite[Lemma~2.8 and Lemma~2.9]{colorado2} we have
\begin{equation}\label{2.8}
\lim_{\delta\rightarrow0}\lim_{j\rightarrow +\infty}\left|\int_{\mathbb{R}^n}u_j(x)(-\Delta)^{s/2}u_j(x)(-\Delta)^{s/2}\psi_{\delta,i_0}(x)dx\right|=0
\end{equation}
and
\begin{equation}\label{2.9}
\lim_{\delta\rightarrow0}\lim_{j\rightarrow +\infty}\left|\int_{\mathbb{R}^n}(-\Delta)^{s/2}u_j(x)\int_{\mathbb{R}^n}\frac{(u_j(x)-u_j(y))
(\psi_{\delta,i_0}(x)-\psi_{\delta,i_0}(y))}{\left|x-y\right|^{n+s}}dxdy\right|=0.
\end{equation}
Then, by combining \eqref{that2}--\eqref{2.9} and \eqref{convergenza misure}--\eqref{nu} we get
\begin{equation}\label{that3}
\lim_{\delta\rightarrow0}\lim_{j\rightarrow +\infty}\iint_{\mathbb R^n \times \mathbb R^n}
\frac{\big(u_j(x)-u_j(y)\big)\big(\psi_{\delta,i_0}(x)u_j(x)-\psi_{\delta,i_0}(y)u_j(y)\big)}
{\left|x-y\right|^{n+2s}}dxdy\geq\mu_{i_0}.
\end{equation}

While, by \eqref{sottocritico} and the Dominated Convergence Theorem we get
\begin{equation*}
\int_{B(x_{i_0},2\delta)} f(x, u_j(x))u_j(x)\psi_{\delta,i_0}(x)dx\rightarrow\int_{B(x_{i_0},2\delta)} f(x, u(x))u(x)\psi_{\delta,i_0}(x)dx\quad\mbox{as }j\rightarrow +\infty,
\end{equation*}
and so by sending $\delta\rightarrow0$ we observe that
\begin{equation}\label{term f}
\lim_{\delta\rightarrow0}\lim_{j\rightarrow +\infty}\int_{B(x_{i_0},2\delta)} f(x, u_j(x))u_j(x)\psi_{\delta,i_0}(x)dx=0.
\end{equation}

Furthermore, by \eqref{convergenza misure} it follows that
	\[\int_\Omega \left|u_j(x)\right|^{2^*}\psi_{\delta,i_0}(x)dx\to\int_\Omega \psi_{\delta,i_0}(x)d\nu\quad\mbox{as }j\to +\infty\,.
\]

Finally, by combining this last formula with \eqref{that is}, \eqref{that3} and \eqref{term f} we get
\begin{equation*}
\nu_{i_0}\geq\frac{\mu_{i_0}}{\gamma}.
\end{equation*}
Thus, from this and \eqref{mu} with $i=i_0$ we have that
$$\nu_{i_0}\geq \frac{\nu_{i_0}^{2/2^*}S(n,s)}{\gamma}\,,$$
which yields that
either $\nu_{i_0}=0$ or $\nu_{i_0}$ verifies \eqref{step}. This ends the proof of Step~\ref{step1}.
\end{proof}

\begin{step}\label{s2}
Estimate \eqref{step} can not occur, hence $\nu_{i_0}=0$.
\end{step}
\begin{proof}
For this, it is enough to see that
\begin{equation}\label{claim2}
\int_\Omega d\nu<\left[\frac{S(n,s)}{\gamma}\right]^{\frac{n}{2s}}\,.
\end{equation}
For this, let us consider two cases. First of all, assume that
\begin{equation}\label{v<1}
{\displaystyle \int_\Omega d\nu\leq 1}\,.
\end{equation}
Since $\gamma<\gamma^*$ and by \eqref{gamma} (which implies that $\gamma^*<S(n,s)$) we have
	\[1<\left(\frac{S(n,s)}{\gamma}\right)^{\frac{n}{2s}},
\]
from which immediately follows \eqref{claim2}, thanks to \eqref{v<1}.

Now, assume that ${\displaystyle \int_\Omega d\nu>1}$. Since $\left\{u_{j}\right\}_{j\in\mathbb{N}}$ is a $(PS)_c$ sequence for $\mathcal J_\gamma$, arguing as in Lemma~\ref{bound} (see formula~\eqref{3.11}) we get
\begin{equation}\label{3.11'}
\mathcal J_\gamma(u_j)-\frac{1}{2}\mathcal J'_\gamma(u_j)(u_j)\ge\frac{s\gamma}{n}\left\|u_j\right\|^{2^*}_{2^*}
-a_1\left|\Omega\right|-a_2\left|\Omega\right|^{\frac{2^*-\sigma}{2^*}}\left\|u_j\right\|^{\sigma}_{2^*}\,.
\end{equation}

By sending $j\rightarrow +\infty$ in \eqref{3.11'} and using \eqref{ps1}, \eqref{convergenza misure} we obtain
$$\begin{aligned}
\frac{s\gamma}{n}\int_\Omega d\nu & \leq c+a_1\left|\Omega\right|+a_2\left|\Omega\right|^{\frac{2^*-\sigma}{2^*}}\Big(\int_\Omega d\nu\Big)^{\frac{\sigma}{2^*}}\\
& \leq\left(M+a_1\left|\Omega\right|+a_2\left|\Omega\right|^{\frac{2^*-\sigma}{2^*}}\right)\Big(\int_\Omega d\nu\Big)^{\frac{\sigma}{2^*}}\\
& = (M+A)\Big(\int_\Omega d\nu\Big)^{\frac{\sigma}{2^*}}\,,
\end{aligned}$$
thanks to the choice of $c\leq M$ and the definition of $A$ given in \eqref{A}. Hence we get
\begin{equation}\label{addraffy222}
\int_\Omega d\nu\leq \left[\frac{n(M+A)}{s\gamma}\right]^{\frac{2^*}{2^*-\sigma}}\,.
\end{equation}

By \eqref{gamma} and the fact that $\gamma<\gamma^*$ we know that
$$\gamma < \left[\left(S(n,s)\right)^{\frac{n}{2s}}\left(\frac{s}{n(M+A)}\right)^
{\frac{2^*}{2^*-\sigma}}\right]^{\frac{1}{n/2s-2^*/(2^*-\sigma)}}\,,$$	
that is
$$\gamma^{\frac{n}{2s}-\frac{2^*}{2^*-\sigma}} < \left(S(n,s)\right)^{\frac{n}{2s}}\left(\frac{s}{n(M+A)}\right)^
{\frac{2^*}{2^*-\sigma}}\,,$$	
which yields
\[\left[\frac{n(M+A)}{s\gamma}\right]^{\frac{2^*}{2^*-\sigma}}
<\left(\frac{S(n,s)}{\gamma}\right)^{\frac{n}{2s}}\,.
\]
From this and \eqref{addraffy222}
we get \eqref{claim2}. Thus, the proof of Step~\ref{s2} is complete and $\nu_{i_0}=0$.
\end{proof}

\begin{step}\label{s3}
Claim \eqref{claim} holds true.
\end{step}
\begin{proof}
By considering that $i_0$ was arbitrary in Step~\ref{step1}, we deduce that $\nu_i=0$ for any $i\in J$. As a consequence, also from \eqref{convergenza misure} and \eqref{nu} it follows that
$u_j\to u$ in $L^{2^*}(\Omega)$ as $j\rightarrow +\infty$.
Thus, by \eqref{sottocritico}, the fact that
\begin{equation}\label{ps1ADD}
\mathcal J'_\gamma(u_j)\rightarrow0\quad\mbox{as}\;j\rightarrow +\infty
\end{equation}
(being $\left\{u_{j}\right\}_{j\in\mathbb{N}}$ a $(PS)_c$ sequence for $\mathcal J_\gamma$)
and the Dominated Convergence Theorem, we have
\begin{equation}\label{4.9}
\lim_{j\to +\infty}\left\|u_j\right\|^2=\gamma\int_\Omega\left|u(x)\right|^{2^*}dx+\int_\Omega f(x,u(x))u(x)dx.
\end{equation}

Moreover, by remembering that $u_j\rightharpoonup u$ in $X_0^s(\Omega)$ and by using again \eqref{sottocritico}, \eqref{ps1ADD}
and the Dominated Convergence Theorem, we have
\begin{equation}\label{4.10}
\left\langle u,\varphi\right\rangle=\gamma\int_\Omega \left|u(x)\right|^{2^*-2}u(x)\varphi(x)dx+\int_\Omega f(x, u(x))\varphi(x)dx,
\end{equation}
for any $\varphi\in X_0^s(\Omega)$. Thus, by combining \eqref{4.9} and \eqref{4.10} with $\varphi=u$ we get the claim \eqref{claim}, concluding the proof of Step~\ref{s3}.
\end{proof}
Hence, the proof of Lemma~\ref{palais} is complete.
\end{proof}

\section{Main theorems}\label{sec main}
This section is devoted to the proof of the main results of the paper. In particular here we study the geometry of the functional~$\mathcal J_\gamma$.

At first, we need some notation. In what follows $\big\{ \lambda_j\big\}_{{j\in\NN}}$  denotes the sequence of the eigenvalues of the following problem
\begin{equation}\label{problemaautovalori}
\left\{\begin{array}{ll}
(-\Delta)^s u=\lambda\, u & \mbox{in } \Omega\\
u=0 & \mbox{in } \RR^n\setminus \Omega,
\end{array}\right.
\end{equation}
with
\begin{equation}\label{lambdacrescente}
0<\lambda_1<\lambda_2\le \dots \le \lambda_j\le \lambda_{j+1}\le \dots
\end{equation}
$${\mbox{$\lambda_j\to +\infty$ as $ j\to +\infty,$}}$$
and with $e_j$ as eigenfunction corresponding to $\lambda_j$.
Also, we choose $\big\{e_j\big\}_{j\in\NN}$ normalized in such a way that
this sequence provides an orthonormal basis of $L^2(\Omega)$ and an orthogonal basis of $X_0^s(\Omega)$. For a complete study of the spectrum of the fractional Laplace operator~$(-\Delta)^s$ we refer to \cite[Proposition~2.3]{sY}, \cite[Proposition~9 and Appendix~A]{svlinking} and \cite[Proposition~4]{servadeivaldinociBNLOW}.

Along the paper, for any $j\in \NN $ we also set
	\[\mathbb{P}_{j+1}=\left\{u\in X_0^s(\Omega):\,\,\left\langle u,e_i\right\rangle=0\quad\mbox{for any } i=1,\ldots,j \right\}\quad(\mbox{with }\mathbb{P}_1=X_0^s(\Omega)),
\]
as defined also in \cite[Proposition~9 and Appendix~A]{svlinking}, while $$\mathbb H_j=\mbox{span}\left\{e_1,\ldots,e_j\right\}$$ will denote the linear subspace generated by the first $j$ eigenfunctions of $(-\Delta)^s$.
It is immediate to observe that $\mathbb P_{j+1}=\mathbb H^{\bot}_j$ with respect to the scalar product in $X_0^s(\Omega)$ defined as in formula~\eqref{prodottoscalare}. Thus, since $X_0^s(\Omega)$ is a Hilbert space (see \cite[Lemma~7]{svmountain} and \eqref{prodottoscalare}), we can write it as a direct sum as follows
$$X_0^s(\Omega)=\mathbb H_j\oplus\mathbb P_{j+1}$$
for any $j\in \NN$\,.
Moreover, since $\big\{e_j\big\}_{j\in\NN}$ is an orthogonal basis of $X_0^s(\Omega)$, it is easy to see that for any $j\in \NN$
$$\mathbb P_{j+1}=\overline{\mbox{span}\left\{e_i:\,\,i\geq j+1\right\}}.$$

Now, before studying and proving the geometric features for $\mathcal J_\gamma$ we need a stronger version of the classical Sobolev embedding. Here the constant of the embedding can be chosen and controlled a priori.

\begin{lemma}\label{lem4.1}
Let $r\in[2,2^*)$ and $\delta>0$.

Then, there exists $j\in\mathbb{N}$ such that $\left\|u\right\|^r_r\leq\delta\left\|u\right\|^r$ for any $u\in\mathbb{P}_{j+1}$.
\end{lemma}
\begin{proof}
By contradiction, we suppose that there exists $\delta>0$ such that for any $j\in\mathbb{N}$ there exists $u_j\in\mathbb{P}_{j+1}$ which verifies $\left\|u_j\right\|^r_r>\delta\left\|u_j\right\|^r$. Considering $v_j=u_j/\left\|u_j\right\|_r$, we have that $v_j\in \mathbb{P}_{j+1}$,
\begin{equation}\label{normavjr}
\left\|v_j\right\|_r=1
\end{equation}
and $\left\|v_j\right\|<1/\delta$ for any $j\in\mathbb{N}$. Thus, the sequence $\left\{v_j\right\}_{j\in\mathbb{N}}$ is bounded in $X_0^s(\Omega)$ and we may suppose that there exists $v\in X_0^s(\Omega)$ such that, up to a subsequence,
$$v_j\rightharpoonup v \quad \mbox{in} \quad X_0^s(\Omega)$$
and
\begin{equation}\label{addconvr}
v_j\to v \quad \mbox{in} \quad L^r(\Omega)
\end{equation}
as $j\to +\infty$. Hence, by \eqref{normavjr} and \eqref{addconvr} we deduce that
\begin{equation}\label{normavr=1}
\left\|v\right\|_r=1\,.
\end{equation}
Moreover, since $\left\{e_j\right\}_{j\in\mathbb{N}}$ is an orthogonal basis of $X_0^s(\Omega)$ by \cite[Proposition~9]{svlinking}, we can write $v$ as follows
$$v=\sum^\infty_{j=1}\left\langle v, e_j\right\rangle e_j\,.$$

Now, given $k\in\mathbb{N}$ we have $\left\langle v_j, e_k\right\rangle=0$ for any $j\geq k$, since $v_j\in\mathbb{P}_{j+1}$. From this we deduce that $\left\langle v, e_k\right\rangle=0$ for any $k\in\mathbb N$, which clearly implies that $v\equiv 0$. On the other hand, this contradicts \eqref{normavr=1}. Hence, Lemma~\ref{lem4.1} holds true.
\end{proof}

\subsection{Geometric setting for Theorem~\ref{a}}\label{sec a}
In order to prove Theorem~\ref{a}, we just have to verify that the energy functional $\mathcal J_\gamma$ satisfies $(I_2)$ and $(I_3)$ of Theorem~\ref{abs}. For this we will consider $V= \mathbb H_j$ and $X=\mathbb{P}_{j+1}$, with $j\in\mathbb{N}$ chosen as in the following result:

\begin{lemma}\label{lem4.2}
Let $f$ satisfy \eqref{f4}.

Then, there exist $\widetilde{\gamma}>0$, $j\in\mathbb{N}$ and $\rho$, $\alpha>0$ such that $\mathcal J_\gamma(u)\geq\alpha$, for any $u\in\mathbb{P}_{j+1}$ with $\left\|u\right\|=\rho$, and $0<\gamma<\widetilde{\gamma}$.
\end{lemma}
\begin{proof}
Take $\gamma>0$.
By \eqref{f4} and \cite[Lemma~6]{svmountain} we get a suitable constant $c>0$ such that
\begin{equation}\label{add22222}
\mathcal J_\gamma(u)\geq\frac{1}{2}\left\|u\right\|^2-b_1\left\|u\right\|^\theta_\theta-b_2\left|\Omega\right|-\gamma c\left\|u\right\|^{2^*},
\end{equation}
for any $u\in X_0^s(\Omega)$.
Let $\delta>0$: we will fix it in the sequel. By \eqref{add22222} and Lemma~\ref{lem4.1} there exists $j\in\mathbb{N}$ such that
\begin{equation}\label{geo1}
\mathcal J_\gamma(u)\geq\left\|u\right\|^2\left(\frac{1}{2}-b_1\delta\left\|u\right\|^{\theta-2}\right)-b_2\left|\Omega\right|-\gamma c\left\|u\right\|^{2^*},
\end{equation}
for any $u\in\mathbb{P}_{j+1}$.

Now, consider $\left\|u\right\|=\rho=\rho(\delta)$, with $\rho$ such that $b_1\delta\rho^{\theta-2}=1/4$, so that
$$\mathcal J_\gamma(u)\geq \frac 1 4 \rho^2-b_2|\Omega|-\gamma c\rho^{2^*}$$
for any $u\in\mathbb{P}_{j+1}$, thanks to \eqref{geo1}.

Now, observe that $\rho(\delta)\rightarrow +\infty$ as $\delta\rightarrow0$, since $\theta>2$. Hence, we can choose $\delta$ sufficiently small such that $\rho^2 /4-b_2\left|\Omega\right|\geq\rho^2 /8$, which  yields
\[\mathcal J_\gamma(u)\geq\frac{1}{8}\rho^2-\gamma c\rho^{2^*},
\]
for any $u\in\mathbb{P}_{j+1}$ with $\left\|u\right\|=\rho$.

Finally, let $\widetilde{\gamma}>0$ be such that $\frac{1}{8}\rho^2-\widetilde{\gamma} c\rho^{2^*}=\alpha>0$. Then we get
$$\mathcal J_\gamma(u)\geq \mathcal J_{\widetilde{\gamma}}(u)\geq \alpha$$
for any $u\in\mathbb{P}_{j+1}$ with $\|u\|=\rho$ and any $\gamma\in (0,\widetilde{\gamma})$\,,
concluding the proof.
\end{proof}

\begin{lemma}\label{lem4.3}
Let $f$ satisfy \eqref{f5} and let $l\in\mathbb{N}$.

Then, there exist a subspace $W$ of $X_0^s(\Omega)$ and a constant $M_l>0$, independent of $\gamma$, such that $\mbox{dim } W=l$ and ${\displaystyle \max_{u\in W}\mathcal J_0(u)<M_l}$.
\end{lemma}
\begin{proof}
Here we can argue exactly as in \cite[Lemma~4.3]{silva} where the classical case of the Laplacian was considered. For this, we can use also the properties of eigenfunctions of $(-\Delta)^s$ (see \cite{svlinking}).
\end{proof}

\begin{proof}[\bf Proof of Theorem~\ref{a}]
By Lemma~\ref{lem4.2} we find $j\in\mathbb{N}$ and $\widetilde{\gamma}>0$ such that $\mathcal J_\gamma$ satisfies $(I_2)$ in $X=\mathbb P_{j+1}$, for any $0<\gamma<\widetilde{\gamma}$. While, by Lemma~\ref{lem4.3} for any $k\in \NN$ there is a subspace $W\subset X_0^s(\Omega)$ with $\mbox{dim }W=k+j$ and such that $\mathcal J_\gamma$ satisfies $(I_3)$ with $M=M_{j+k}>0$ for any $\gamma>0$, since $\mathcal J_\gamma<\mathcal J_0$.

Finally, we note that by Lemma~\ref{palais}, considering $\widetilde{\gamma}$ smaller if necessary, we have that $\mathcal J_\gamma$ satisfies $(I_4)$ for any $0<\gamma<\widetilde{\gamma}$. Thus, we may apply Theorem~\ref{abs} to conclude that $\mathcal J_\gamma$ admits $k$ pairs of non--trivial critical points for $\gamma>0$ sufficiently small. Hence, Theorem~\ref{a} is proved.
\end{proof}

\subsection{Geometric setting for Theorem~\ref{b}}\label{sec b}
We apply again Theorem~\ref{abs} to the functional~$\mathcal J_\gamma$. By considering $\lambda_j\leq\lambda_k$ as in \eqref{ft5} and \eqref{ft4}, we have two cases. When $j=1$ we set $V=\left\{0\right\}$, so $X=X_0^s(\Omega)$: note that this is consistent with the situation $\mathbb{P}_1=X_0^s(\Omega)$. While if $j>1$ we consider $X=\mathbb{P}_j$ and $V=\mathbb H_{j-1}$. Moreover, we set $W=\mathbb H_k$ as subspace of $X_0^s(\Omega)$ in $(I_3)$.

Now, in order to verify the geometric assumptions $(I_2)$ and $(I_3)$ in Theorem~\ref{abs} we consider here two different characterizations of the eigenvalues of $(-\Delta)^s$. That is, for any $j\in\mathbb{N}$ by \cite[Proposition~9]{svlinking} we have that
\begin{equation}\label{min}
\lambda_j=\min_{u\in\mathbb{P}_j\setminus\{0\}}\frac{\left\|u\right\|^2}{\left\|u\right\|^2_2},
\end{equation}
while from \cite[Proposition~2.3]{sY} we know that
\begin{equation}\label{max}
\lambda_j=\max_{u\in\mathbb H_j\setminus\{0\}}\frac{\left\|u\right\|^2}{\left\|u\right\|^2_2}.
\end{equation}

Moreover, we need the following technical lemma:

\begin{lemma}\label{lemtec}
Let $a:\Omega\rightarrow\mathbb{R}$ be the measurable function given in \eqref{ft5}. Then, there exists $\beta>0$ such that for any $u\in\mathbb{P}_j$
$$\left\|u\right\|^2-\int_\Omega a(x)\left|u(x)\right|^2dx\geq\beta\left\|u\right\|^2_2\,.$$
\end{lemma}
\begin{proof}
We argue by contradiction and we suppose that for any $i\in\mathbb{N}$ there exists $u_i\in\mathbb{P}_j$ such that
\begin{equation}\label{AddRaffy}
\left\|u_i\right\|^2-\int_\Omega a(x)\left|u_i(x)\right|^2dx<\frac{1}{i}\left\|u_i\right\|^2_2.
\end{equation}

Let $v_i=u_i/\left\|u_i\right\|_2$. Of course, $v_i\in \mathbb{P}_j$ and
\begin{equation}\label{vi}
\left\|v_i\right\|_2=1
\end{equation}
for any $i\in \NN$. By \eqref{ft5}, \eqref{min}, \eqref{AddRaffy} and \eqref{vi} we get
\begin{equation}\label{4.5}
\begin{aligned}
\lambda_j & \leq\left\|v_i\right\|^2\\
& <\int_\Omega a(x)\left|v_i(x)\right|^2 dx+\frac{1}{i}\\
& \leq\lambda_j\int_\Omega \left|v_i(x)\right|^2 dx+\frac{1}{i}\\
& \leq\lambda_j+\frac{1}{i}
\end{aligned}
\end{equation}
for any $i\in \NN$. From this, we have that $\left\{v_i\right\}_{i\in\mathbb{N}}$ is a bounded sequence in $X_0^s(\Omega)$. Therefore, by applying \cite[Lemma~8]{svmountain} and \cite[Theorem~IV.9]{brezis} there exists $v\in X_0^s(\Omega)$ such that, up to a subsequence, $v_i$ converges to $v$ weakly in $X_0^s(\Omega)$, strongly in $L^2(\Omega)$ and a.e. in $\Omega$ as $j\to +\infty$ and $\left|v_i\right|\leq h\in L^2(\Omega)$ a.e. in $\Omega$. Thus, by \eqref{vi} we know that $\left\|v\right\|_2=1$, so that $v$ is almost everywhere different from zero in $\Omega$, i.e.
\begin{equation}\label{vnot0}
v\not \equiv 0 \quad \mbox{in} \quad \Omega\,.
\end{equation}

By sending $i\rightarrow +\infty$ in \eqref{4.5} and using the Dominated Convergence Theorem and \eqref{AddRaffy}, we get
\begin{equation}\label{4.6}
\int_\Omega \big(\lambda_j-a(x)\big)\left|v(x)\right|^2 dx=0.
\end{equation}
Then, \eqref{ft5}, \eqref{vnot0} and \eqref{4.6} implies that
\[a(x)=\lambda_j\quad\mbox{a.e. in }\Omega, \]
which contradicts the assumption \eqref{ft5}. Hence, Lemma~\ref{lemtec} holds true.
\end{proof}

Now we are ready to prove that $\mathcal J_\gamma$ satisfies $(I_2)$ and $(I_3)$ of Theorem~\ref{abs}.
\begin{lemma}\label{lem5.1}
Let $f$ satisfy \eqref{f1}, \eqref{f2} and \eqref{ft5}.

Then, for any $\gamma>0$ there exist $\rho$, $\alpha>0$ such that $\mathcal J_\gamma(u)\geq\alpha$ for any $u\in\mathbb{P}_j$ with $\left\|u\right\|=\rho$.
\end{lemma}
\begin{proof}
Fix $\gamma>0$. By \eqref{f1}, \eqref{f2} and \eqref{ft5}, for any $\varepsilon>0$ there exists $C_\varepsilon>0$ such that
\begin{equation}\label{4.7}
\left|F(x,t)\right|\leq\frac{C_\varepsilon}{2^*}\left|t\right|^{2^*}+\frac{a(x)+\varepsilon}{2}\left|t\right|^2,
\end{equation}
for any $t\in\mathbb{R}$ and a.e. $x\in\Omega$.

Now, let $\beta>0$ be as in Lemma~\ref{lemtec} and $\varepsilon'>0$ be such that $\beta-\varepsilon'\lambda_j>0$. Thus, by \eqref{ft5} and Lemma~\ref{lemtec}, we have
\begin{equation*}
\begin{aligned}
\left\|u\right\|^2-\int_\Omega a(x)\left|u(x)\right|^2 dx
&=\frac{1+\varepsilon'}{1+\varepsilon'}\left(\left\|u\right\|^2-\int_\Omega a(x)\left|u(x)\right|^2 dx\right)\\
&=\frac{\varepsilon'}{1+\varepsilon'}\left\|u\right\|^2+\frac{1}{1+\varepsilon'}\left(\left\|u\right\|^2-\int_\Omega a(x)\left|u(x)\right|^2 dx-\varepsilon'\int_\Omega a(x)\left|u(x)\right|^2 dx\right)\\
&\geq\frac{\varepsilon'}{1+\varepsilon'}\left\|u\right\|^2+\frac{1}{1+\varepsilon'}\left(\beta\left\|u\right\|^2_2-\varepsilon'\int_\Omega a(x)\left|u(x)\right|^2 dx\right)\\
&\geq\frac{\varepsilon'}{1+\varepsilon'}\left\|u\right\|^2+\int_\Omega(\beta-\varepsilon'\lambda_j)\left|u(x)\right|^2 dx\\
&\geq\frac{\varepsilon'}{1+\varepsilon'}\left\|u\right\|^2
\end{aligned}
\end{equation*}
for any $u\in\mathbb{P}_j$.
From this and by \eqref{4.7} we get
\begin{equation*}
\begin{alignedat}2
\mathcal J_\gamma(u)&=\frac 1 2 \left\|u\right\|^2-\frac{\gamma}{2^*}\left\|u\right\|^{2^*}_{2^*}
-\int_\Omega F(x,u(x))\,dx\\
&\geq\frac{1}{2}\left(\left\|u\right\|^2-\int_\Omega a(x)\left|u(x)\right|^2 dx\right)-\frac{1}{2^*}(\gamma+C_\varepsilon)\left\|u\right\|^{2^*}_{2^*}-\frac{\varepsilon}{2}\left\|u\right\|^2_2\\
&\geq\frac{\varepsilon'}{2(1+\varepsilon')}\left\|u\right\|^2-\frac{1}{2^*}(\gamma+C_\varepsilon)\left\|u\right\|^{2^*}_{2^*}-\frac{\varepsilon}{2}\left\|u\right\|^2_2
\end{alignedat}
\end{equation*}
for any $u\in\mathbb{P}_j$.
Thus, by \cite[Lemma~6]{svmountain} and taking $\varepsilon>0$ sufficiently small, there exist constants $K$, $C>0$ such that
\begin{equation}\label{bisci}
\mathcal J_\gamma(u)\geq K\rho^2-C\rho^{2^*}
\end{equation}
for any $u\in\mathbb{P}_j$ with $\left\|u\right\|=\rho$. By taking $\rho>0$ small enough, \eqref{bisci} gives that
$$\mathcal J_\gamma(u)\geq \alpha$$
for a suitable $\alpha>0$, since $2^*>2$.
\end{proof}

\begin{lemma}\label{lem5.2}
Let $f$ satisfy \eqref{ft4}.

Then, for any $\gamma>0$ there exists a constant $M>0$, independent of $\gamma$, such that ${\displaystyle \max_{u\in\mathbb H_k}\mathcal J_\gamma(u)<M}$.
\end{lemma}
\begin{proof}
Fix $\gamma>0$.
By \eqref{ft4} and \eqref{max}, for any $u\in\mathbb H_k\setminus\{0\}$ we have
$$\begin{aligned}
\mathcal J_\gamma(u) & \leq\frac{1}{2}\left\|u\right\|^2-\frac{\lambda_k}{2}\left\|u\right\|^2_2-
\frac{\gamma}{2^*}\left\|u\right\|^{2^*}_{2^*}+B\left|\Omega\right|\\
& \leq B\left|\Omega\right|-\frac{\gamma}{2^*}\left\|u\right\|^{2^*}_{2^*}\\
&<B\left|\Omega\right|,
\end{aligned}$$
concluding the proof of Lemma~\ref{lem5.2}.
\end{proof}

\begin{proof}[\bf Proof of Theorem~\ref{b}]
By Lemma~\ref{palais}, Lemma~\ref{lem5.1} and Lemma~\ref{lem5.2}, there is $\gamma^*>0$ sufficiently small such that $\mathcal J_\gamma$ satisfies $(I_2)-(I_4)$ of Theorem~\ref{abs} for any $\gamma\in(0,\gamma^*)$. By recalling that $\mathbb{P}_j=H^\bot_{j-1}$, we get that $\mbox{codim}~\mathbb{P}_j=j-1$. Hence, by Theorem~\ref{abs} we conclude that $\mathcal J_\gamma$ admits $k-j+1$ pairs of non--trivial critical points for any $\gamma\in(0,\gamma^*)$. Then, the proof of Theorem~\ref{b} is complete.
\end{proof}

\begin{remark}
We would like to point out that when $j=1$ we can also replace \eqref{ft4} with \eqref{f5} and Theorem~\ref{b} still holds true. Indeed, we can argue exactly as in the proof of Theorem~\ref{a}, by using Lemma~\ref{lem5.1} $($with $\mathbb{P}_1=X_0^s(\Omega)$$)$ instead of Lemma~\ref{lem4.2}.
\end{remark}

\subsection{Proof of Theorem~\ref{c}}\label{sec c}

We first show that problem~\eqref{P} possesses a non--trivial non--negative solution. For this, it is sufficient to study the following problem
\begin{equation}\label{Ptronc}
\left\{\begin{array}{lll}
$$(-\Delta)^su=\gamma u^{2^*-1}+\widetilde{f}(x,u) & \mbox{in } \Omega$$\\

$$u\geq0 & \mbox{in } \Omega$$\\

$$u=0 & \mbox{in } \mathbb{R}^{n}\setminus\Omega\,,$$
\end{array}
\right.
\end{equation}
where
\begin{equation}\label{ftilde}
\widetilde{f}(x,t)=\begin{cases}
f(x,t) & \mbox{if } t>0\\
0 & \mbox{if } t\leq0.
\end{cases}
\end{equation}
Indeed, a non-trivial solution of \eqref{Ptronc} is a non-trivial non-negative solution of \eqref{P}.

The energy functional associated with \eqref{Ptronc} is given by
\begin{equation}\label{Jtilde}
\widetilde{\mathcal J}_\gamma(u)=\frac 1 2 \left\|u\right\|^2-\frac{\gamma}{2^*}\int_\Omega (u(x))^{2^*}\,dx
-\int_\Omega \widetilde{F}(x,u(x))\,dx,
\end{equation}
where
$$\widetilde{F}(x,t)=\int^t_0\widetilde{f}(x,\tau)d\tau\,.$$

We would observe that the truncated function $\widetilde{f}$ still verifies \eqref{f1}, \eqref{f2}, \eqref{f3} and \eqref{ft5}, while \eqref{ft4} holds true for $\widetilde{f}$ for any $t\geq0$ but not for any $t<0$. This point must be considered for our proof.

Indeed, in order to apply Theorem~\ref{abs2}, we immediately note that $\widetilde{\mathcal J}_\gamma$ still verifies $(I_2)$ by Lemma~\ref{lem5.1} with $\mathbb{P}_1=X_0^s(\Omega)$. In order to prove $\widehat{(I_3)}$ of Theorem~\ref{abs2} we have to proceed as follows.

Let $e_1$ be the eigenfunction of $(-\Delta)^s$ associated to $\lambda_1$. Since $e_1$ is positive by \cite[Corollary~8]{servadeivaldinociRego}, by \eqref{ftilde} it follows that $\widetilde{F}(x,te_1(x))=F(x,te_1(x))$ for any $t>0$ and for a.e. $x\in\Omega$. Thus, we can use \eqref{ft4} and get for any $t>0$
$$\begin{aligned}
\widetilde{\mathcal J}_\gamma(te_1)& = \frac 1 2 \left\|te_1\right\|^2-\frac{\gamma}{2^*}\left\|te_1\right\|^{2^*}_{2^*}
-\int_\Omega\widetilde{F}(x,te_1(x))\,dx\\
& \leq\frac{t^2}{2}\left\|e_1\right\|^2-\frac{t^2}{2}\lambda_1\left\|e_1\right\|^2_2+B\left|\Omega\right|\\
&=B\left|\Omega\right|,
\end{aligned}$$
thanks to the characterization of $e_1$ given in \cite[Proposition~9]{svlinking}. From this, $\widetilde{\mathcal J}_\gamma$ satisfies $\widehat{(I_3)}$ for any $\gamma>0$.

Now it remains to verify $(I_4)$ of Theorem~\ref{abs2}: for this it is enough to argue as in the proof of Lemma~\ref{bound} and Lemma~\ref{palais} (note that for these lemmas we just need assumptions \eqref{f1}, \eqref{f2} and \eqref{f3}).

Finally, all the assumptions of Theorem~\ref{abs2} are satisfied by $\widetilde{\mathcal J}_\gamma$ and so we can conclude that for any $\gamma\in(0,\gamma^*)$, $\widetilde{\mathcal J}_\gamma$ has a non--trivial critical point which is a non--trivial non--negative solution for \eqref{P}. In a similar way, with small modifications, it is possible to prove the existence of a non--trivial non--positive solution for \eqref{P}. This ends the proof of Theorem~\ref{c}.


\begin{thebibliography}{99}

\bibitem{adams} {\sc R.A. Adams}, {\em Sobolev Spaces}, Academic Press, New York, 1975.

\bibitem{ar} {\sc A. Ambrosetti and P. Rabinowitz},
{\em Dual variational methods in critical point theory and
applications}, J. Funct. Anal., 14, 349--381 (1973).

\bibitem{colorado2} {\sc B. Barrios, E. Colorado, R. Servadei and F. Soria}, {\em A critical fractional
equation with concave-convex power nonlinearities}, Ann. Inst. H. Poincar\'e Anal. Non Lin\'eaire, 32, 875--900, (2015).

\bibitem{bartolo}{\sc P. Bartolo, V. Benci and D. Fortunato}, {\em Abstract critical point theorems and applications to some nonlinear problems with strong resonance at infinity},
Nonlinear Anal., 7, 981--1012 (1983).

\bibitem{bartolomolica}{\sc R. Bartolo and G. Molica Bisci}, {\em Asymptotically linear fractional $p$-Laplacian equations}, to appear in Annali Mat. Pura Appl.

\bibitem{brezis} {\sc H. Br\'ezis}, Analyse fonctionelle. Th\'{e}orie et
applications, {\em Masson}, Paris (1983).


\bibitem{valpal}{\sc E. Di Nezza, G. Palatucci and E. Valdinoci}, {\em Hitchhiker's guide to the fractional Sobolev spaces}, Bull. Sci. Math., 7, 981--1012 (1983).

\bibitem{fms} {\sc A. Fiscella, G. Molica Bisci and R. Servadei}, {\em Bifurcation and multiplicity results for critical nonlocal fractional Laplacian problems}, Bull. Sci. Math., 140, 14--35 (2016).


\bibitem{fsv} {\sc A. Fiscella, R. Servadei and E. Valdinoci}, {\em Density properties for fractional Sobolev spaces}, Ann. Acad. Sci. Fenn. Math., 40, 235--253 (2015).

\bibitem{fv} {\sc A. Fiscella and E. Valdinoci},  {\em A critical Kirchhoff type problem involving a nonlocal operator}, Nonlinear Anal., 94, 156--170 (2014).

\bibitem{fp} {\sc G. Franzina and G. Palatucci},  {\em Fractional $p$--eigenvalues}, Riv. Mat. Univ. Parma, 5, 373--386 (2014).

\bibitem{ll} {\sc E. Lindgren and P. Lindqvist}, {\em Fractional eigenvalues}, Calc. Var., 49, 795--826 (2014).

\bibitem{ms} {\sc S. Mosconi and M. Squassina}, {\em Nonlocal problems at nearly critical growth}, preprint available online at {\tt http://arxiv.org/abs/1512.01956}\,.

\bibitem{palatuccipisante}{\sc G. Palatucci and A. Pisante}, {\em Improved Sobolev embeddings, profile decomposition, and concentration--compactness for
fractional Sobolev spaces}, Calc. Var. Partial Differential Equations,  50, no.~3-4, 799--829 (2014).

\bibitem{pereira}{\sc K. Pereira, M. Squassina and Y. Yang}, {\em Bifurcation results for critical growth fractional $p$-Laplacian problems}, Math. Nachr., 289, 332--342 (2015).

\bibitem{sY}{\sc R. Servadei}, {\em The Yamabe equation in a non-local setting},
Adv. Nonlinear Anal., 2, 235--270 (2013).

\bibitem{sv}{\sc R. Servadei and E. Valdinoci}, {\em Lewy-Stampacchia type estimates for variational
inequalities driven by (non)local operators}, Rev. Mat. Iberoam., 29, no.~3, 1091--1126 (2013).

\bibitem{svmountain}{\sc R. Servadei and E. Valdinoci}, {\em Mountain Pass solutions
for non-local elliptic operators}, J. Math. Anal. Appl., 389, 887--898 (2012).

\bibitem{svlinking}{\sc R. Servadei and E. Valdinoci}, {\em Variational methods for non-local operators of elliptic type}, Discrete Contin. Dyn. Syst., 33, no.~5, 2105--2137 (2013).

\bibitem{servadeivaldinociBNLOW}{\sc R. Servadei and E. Valdinoci}, {\em A Brezis-Nirenberg result for non-local critical equations in low dimension}, Commun. Pure Appl. Anal., 12, no.~6, 2445--2464 (2013).

\bibitem{servadeivaldinociRego}{\sc R. Servadei and E. Valdinoci}, {\em Weak and viscosity solutions of the fractional Laplace equation}, Publ. Mat., 58, no.~1, 133-154 (2014).

\bibitem{servadeivaldinociBN}{\sc R. Servadei and E. Valdinoci}, {\em The Brezis-Nirenberg result for the fractional Laplacian}, Trans. Amer. Math. Soc., 367, no.~1, 67--102 (2015).

\bibitem{silva}{\sc E.A.B. Silva and M.S. Xavier}, {\em
Multiplicity of solutions for quasilinear elliptic problems involving critical Sobolev exponents}, Ann. Inst. H. Poincar\'e Anal. Non Lin\'eaire, 20, no.~2, 341--358 (2003).

\end{thebibliography}
\end{document}